\documentclass[11pt]{article}
\usepackage{color}
\usepackage{epsfig}
\usepackage{lscape}
\usepackage{amssymb} 
\usepackage{graphicx}
\usepackage{footnote}
\usepackage[utf8]{inputenc}	
\usepackage{mathtools}
\usepackage{amsthm}
\usepackage{wasysym}
\usepackage{amsfonts}
\usepackage[english]{babel}
\usepackage{bigints}
\usepackage{mathrsfs}
\usepackage{lineno}
\usepackage{nicefrac}
\usepackage{enumerate}
\usepackage{appendix}
\usepackage{graphicx}
\usepackage{commath}
\usepackage{multirow}
\usepackage{tabu}

\usepackage{times}
\usepackage[T1]{fontenc}

\usepackage{algorithm} 
\usepackage{algpseudocode} 
\usepackage{booktabs}

\usepackage[most]{tcolorbox}
\usepackage{mdframed}

\usepackage{amsmath}
\usepackage{mathtools} 
\usepackage{mathrsfs}
\usepackage{tabularx}

\usepackage[]{natbib}
\usepackage{csquotes}

\usepackage{authblk}

\usepackage[colorlinks=true,citecolor=blue,linkcolor=black,urlcolor=blue]{hyperref}
\definecolor{light-gray}{gray}{0.91}

\usepackage[margin=2cm]{geometry}

\allowdisplaybreaks[1]
\usepackage{pifont}

\makeindex

\usepackage{titlesec} 
\usepackage{titletoc} 

\definecolor{princetonorange}{rgb}{0.94,0.59,0.035}
\definecolor{ResnickRed}{rgb}{0.7, 0.13, 0.13}
\definecolor{ResnickBlue}{rgb}{0.0, 0.25, 0.55}
\definecolor{beige}{rgb}{0.937,0.894,0.749}

\titleformat{\section}[block]                 
{\normalfont\bfseries\huge\color{ResnickBlue}} 
{\color{ResnickRed}\thesection}             
{1em}                                       
{}                                          

\titleformat{\subsection}[block]               
{\normalfont\bfseries\Large\color{ResnickBlue}} 
{\color{ResnickRed}\thesubsection}           
{1em}                                        
{}                                           

\titleformat{\subsubsection}[block]            
{\normalfont\bfseries\large\color{ResnickBlue}} 
{\color{ResnickRed}\thesubsubsection}        
{1em}                                        
{}                                           

\titleformat{\paragraph}[block]                
{\normalfont\bfseries\color{ResnickRed}}     
{\color{ResnickRed}\theparagraph}             
{1em}                                        
{}                                           

\titlecontents{section}[0em] {\bfseries\color{ResnickRed}\huge}{}{0em}{\thecontentslabel\quad\hspace{1em}}
\titlecontents{subsection}[2em] {\bfseries\color{ResnickRed}\Large}{}{0em}{\thecontentslabel\quad\hspace{1em}}
\titlecontents{subsubsection}[4em] {\bfseries\color{ResnickRed}\large}{}{0em}{\thecontentslabel\quad\hspace{1em}}

\usepackage{authblk}

\newcommand{\imag}{\textbf{i}}

\newcommand{\ImagPart}{\mathfrak{Im}}
\renewcommand{\Im}{\mathfrak{Im}}
\newcommand{\RealPart}{\mathfrak{Re}}
\newcommand{\error}{\mathcal{O}}

\newtheorem{definition}{Definition}[section]
\newtheorem{theorem}{Theorem}[section]

\newtheorem{remark}[theorem]{Remark}

\usepackage{color}

\begin{document}
	
\begin{flushleft}
    \colorbox{black}{\textcolor{white}{\parbox{\dimexpr\textwidth-1.1em}{\vspace{0.2ex}\flushleft
			\setlength{\baselineskip}{2.3em} 
			\textbf{\huge The Complex-Step Integral Transform}\vspace{1ex}}}}
		
	\vspace{1em}  
	
	\setlength{\baselineskip}{\normalbaselineskip} 
	
	\large{Rafael Abreu}\textsuperscript{1,*}, \large{Stephanie Durand}\textsuperscript{2}, \large{Jochen Kamm}\textsuperscript{3}, \large{Christine Thomas}\textsuperscript{4}, \large{Monika Pandey}\textsuperscript{5} \\[0.5em]  
	
	\scriptsize 
	\textsuperscript{1} \textit{Institut de Physique du Globe de Paris, CNRS, Universit\'e de Paris, 75005 Paris, France} \\
	\textsuperscript{2} \textit{Univ Lyon, UCBL, ENSL, UJM, CNRS, LGL-TPE, F-69622, Villeurbanne, France} \\
	\textsuperscript{3} \textit{Geological Survey of Finland, Espoo, Findland} \\
	\textsuperscript{4} \textit{Institut f\"ur Geophysik, Westf\"alische Wilhelms-Universit\"at M\"unster, M\"unster, Germany} \\
	\textsuperscript{5} \textit{Louisiana State University, Baton Rouge, LA, USA} \\[1em]  
	\textsuperscript{*} {email: rabreu@ipgp.fr}
	
	\normalsize 
\end{flushleft}

\begin{abstract}
	Building on the well-established connection between the Hilbert transform and derivative operators, and motivated by recent developments in complex-step differentiation, we introduce the \emph{Complex-Step Integral Transform} (CSIT): a generalized integral transform that combines analytic continuation, derivative approximation, and multi-scale smoothing within a unified framework. A spectral analysis shows that the CSIT preserves phase while suppressing high-wavenumber noise, offering advantages over conventional Fourier derivatives. We discuss the roles of the real and imaginary step parameters, compare FFT-based and interpolation-based implementations, and demonstrate the method on the advection equation and instantaneous-frequency computation. Results show that the CSIT yields smoother, more robust attributes than Hilbert-based methods and provides built-in stabilization for PDE solvers. The CSIT thus represents a flexible alternative for numerical differentiation, spectral analysis, and seismic signal processing. The method opens several avenues for future work, including non-periodic implementations, adaptive parameter selection, and integration with local interpolation frameworks such as high-order Finite-Element methods.
\end{abstract}


\newpage

\section{Introduction}

Integral transforms are central mathematical tools across the physical sciences, providing systematic ways to recast functions into alternative domains where structure, oscillations, and scaling behavior become more transparent. Classical transforms such as the Fourier, Laplace, Radon, and Hilbert transforms underpin a wide range of techniques in signal analysis, inverse problems, and the study of differential equations. 

Within seismology, integral transforms are especially valuable because they allow recorded wavefields—often nonstationary, multicomponent, and noisy—to be analyzed in representations that emphasize their physical content. Among these, the Fourier and Hilbert transforms are arguably the most widely used \citep{bracewell1986fourier, feldman2011hilbert}. The Fourier transform enables spectral analysis, filtering, and the formulation of elastodynamic problems in the wavenumber domain \citep{bracewell1986fourier, slawinski2010waves}, while the Hilbert transform is fundamental for constructing analytic signals, computing instantaneous amplitude and phase, and deriving seismic envelopes \citep[e.g.][]{feldman2011hilbert, bowman2013hilbert, shen2015up}. 

Mathematically, many integral transforms share advantageous properties—linearity, convolution/multiplication dualities, and in many cases analytic continuation—which make them powerful for both theoretical manipulation and efficient numerical implementation (e.g., via the FFT). Practically, however, each transform brings trade-offs: periodicity assumptions, edge effects, sampling and aliasing, and sensitivity to noise must be managed through padding, tapering, filtering or alternative bases (e.g. nonperiodic Chebyshev or localized wavelet bases). These considerations motivate the development of generalized transforms and regularized spectral operators that combine analytic continuation, derivative estimation and multi-scale smoothing for robust seismic attribute extraction.

Several generalizations of the Hilbert transform, such as the fractional, bilinear and trilinear Hilbert transforms are active areas of research  \citep[e.g.][]{tao2008generalization,nabighian2001unification,lohmann1996fractional,venkitaraman2014fractional,cruz2018limited}. These generalizations are applied to problems that range from mechanical vibrations \citep{feldman2011hilbert} to telecommunications  \citep{king_2009} and they are also used in the Earth Sciences to illuminate faults and channels embedded in 3D seismic data \citep{luo2003generalized}, to achieve better noise robustness and higher resolution results to capture subtle geologic features \citep{zhang2019extended}, for automatic interpretation of potential field data \citep{nabighian1984toward}. Also, combining the Hilbert transform with the first and/or second derivative operators has been used to enhance the detection of seismic waves generated in the deep Earth \citep{chambers2005reflectivity,saki2019causes,ibourichene2018detection}.

The most common expression for the Hilbert transform is given by the following equation \citep{king_2009}
\begin{align}
	H(u(t)) = \frac{1}{\pi} \int_{-\infty}^{\infty} \frac{u(\tau)}{t-\tau} d \tau ,
	\label{eq.Hilbert_transform}
\end{align}
where $u(t)$ is the function, of parameter $t$, to be transformed. Eq. \eqref{eq.Hilbert_transform} can be understood as the convolution of the function $u(t)$ with the function $1/(\pi t)$. With a simple change of variables we can write the Hilbert transform (eq. \eqref{eq.Hilbert_transform}) as follows \citep{zygmund2002trigonometric}
\begin{align}
	H(u(t)) =- \frac{1}{\pi} \lim_{\epsilon \to 0}\int_{\epsilon}^{\infty} \frac{u(t+\tau) - u(t-\tau)}{\tau} d \tau .
	\label{eq.FD_Hilbert_transform}
\end{align}
One can note that the term inside the integral has the mathematical structure of the centered finite-difference (FD) approximation for the first-order derivative given by the following expression \citep{Thomas1995,Cohen2001,Trefethen1996,Strikwerda2004,Moczo2014}
\begin{flalign}
	f'(x) = \frac{f(x+\Delta x) - f(x-\Delta x)}{2\Delta x} +\error \left(\Delta x^2 \right) ,
	\label{eq.CenteredFD}
\end{flalign}
where $\Delta x$ refers to an increment in the $x$ direction or grid spacing and $\error (\Delta x^2)$ to the error in the approximation. A generalization of the FD approximation eq. \eqref{eq.CenteredFD} is the complex-step (CS) approximation \citep{Squire1998} given by the following expression
\begin{flalign}
	f'(x) = \frac{\ImagPart [f(x+i \Delta x)]}{\Delta x} + \error \left(\Delta x^2 \right),
	\label{eq.GCS_CSFirstOrder}
\end{flalign}
where $\ImagPart$ refers to the imaginary part and $i^2=-1$. The CS approximation eq. \eqref{eq.GCS_CSFirstOrder} has proven to have several advantages over conventional FD techniques, mainly the numerical instabilities related to subtraction cancellation errors can be avoided (note the subtraction term in the numerator in eq. \eqref{eq.CenteredFD} is absent in eq. \eqref{eq.GCS_CSFirstOrder}) and higher-order accuracy at a single step \citep{Abreu01072015,Abreu201384,anderson2001sensitivity,al2010complex,Voorhees20111146,martins2001connection,martins2003complex,dziatkiewicz2016complex,Jin2010,Kim2006177,Burg2003,Wang2006,ridout2009}.

Several generalizations to the initial CS derivative approximation \eqref{eq.GCS_CSFirstOrder} have been made (e.g. \cite{Lai2008,Cervino2003,Abokhodair2009,Abreu201384,ibrahim2013fractional,Hurkamp2015,Lantoine2012}). A simple generalization made by \cite{Abreu201384} consisted into taking a \textit{complex step} $(h+i v)$ in a strict sense in the original CS approximation \eqref{eq.GCS_CSFirstOrder} as follows
\begin{flalign}
	f'(x) = \frac{\ImagPart[f(x+h+ iv)]}{v} + \error \left(h,v^2\right),
	\label{eq.FirstOrderImag1}
\end{flalign}
where $h,v \in \mathbb{R}$ and $h,v \rightarrow 0$. Although discretization eq. \eqref{eq.FirstOrderImag1} is less accurate compared to the original CS discretization eq. \eqref{eq.GCS_CSFirstOrder}, it allows a generalization of the original CS derivative concept to a wider range of possible higher-order accurate approximations using the imaginary part of the function (for a full list see \cite{Abreu201384,Abreu01072015,ABREU2018390}).

The connection between the Hilbert-transform representation eq. \eqref{eq.FD_Hilbert_transform} and the complex-step derivative \eqref{eq.FirstOrderImag1} naturally motivates a new generalization: the \emph{Complex-Step Integral Transform} (CSIT). In this work, we introduce the CSIT, analyze its mathematical properties, and illustrate its performance for analytic and non-analytic functions. As practical applications, we demonstrate the use of the CSIT for (1) solving the advection equation and (2) computing the instantaneous frequency of seismic signals. In both cases, we highlight the advantages of the CSIT relative to existing techniques. Finally, we summarize the main contributions and discuss potential directions for future work.

\section{The Complex-Step-Integral-Transform }

\begin{definition}[Complex-Step Integral Transform (CSIT)]
	Let $f:\mathbb{R} \to \mathbb{C}$ be a function that admits an analytic continuation 
	into the rectangle
	\[
	\Omega_{H,Z}(x) = \{\, x + \eta + i\tau : -H \le \eta \le H,\; 0 \le \tau \le Z \,\},
	\]
	with $H \in \mathbb{R}$ and $Z \in \mathbb{R}^+$. 
	The \emph{Complex-Step Integral Transform} (CSIT) of $f$ is defined by
	\begin{equation}
		\label{eq.Complex_step_transform}
	(\mathcal{C}_{H,Z}f)(x):=\lim_{\varepsilon\to0^+}\frac{1}{2HZ}\int_{-H}^H\int_{\varepsilon}^Z\frac{\ImagPart[f(x+\eta+i\tau)]}{\tau}\,d\tau\,d\eta.
	\end{equation}
\end{definition}

This transform averages the local complex-step derivative $\Im[f(x+i\tau)]/\tau$ over a finite rectangular domain in the complex plane, providing a stable, multi-scale estimate of the derivative of $f$. Unlike the Hilbert transform eq. \eqref{eq.Hilbert_transform}, the CSIT transform eq. \eqref{eq.Complex_step_transform} allows the function $f(x)$ to be complex valued. Like the Hilbert transform, the physical units of the transformed function are the same as the original function $f(x)$.

\subsection{Existence and Error Bounds}

\begin{theorem}[Existence and Taylor expansion for $H,Z\to 0$]
	\label{thm:csit-symmetric}
	
	Let $f$ be analytic in the rectangular domain
\begin{align}
	\Omega := \{\, x+\eta + i\tau :\; -H\le \eta\le H,\; 0\le\tau\le Z \,\},
\end{align}
	for some $H,Z>0$, and assume there exists a constant \(M>0\) such that
\begin{align}
		\sup_{z\in\Omega} |f^{(3)}(z)| \le M.
\end{align}
 Then the improper double integral in eq. \eqref{eq.Complex_step_transform} exists and the normalized transform
\begin{align}
	(\mathcal{C}_{H,Z}f)(x)
	:= \lim_{\varepsilon\to0^+}\frac{1}{2 H Z}\int_{-H}^{H}\int_{\varepsilon}^{Z}
	\frac{\Im\big[f(x+\eta+i\tau)\big]}{\tau}\,d\tau\,d\eta ,
\end{align}
satisfies the expansion
	\begin{equation}
		(\mathcal{C}_{H,Z}f)(x) \;=\; f'(x) \;+\; \widetilde E(x;H,Z),
		\label{eq:csit-symmetric-expansion}
	\end{equation}
	with the explicit uniform remainder bound
	\begin{equation}
		\big|\widetilde E(x;H,Z)\big| \;\le\; \frac{M}{6}\,H^2 \;+\; \frac{M}{18}\,Z^2 .
		\label{eq:csit-symmetric-remainder}
	\end{equation}
	In particular, the linear-in-\(H\) bias present for one-sided averaging is eliminated by the symmetric \(\eta\)-average.
\end{theorem}

\begin{proof}
	
	Since \(f\) is analytic on \(\Omega\), for each \((\eta,\tau)\in [-H,H]\times[0,Z]\) we may apply Taylor's theorem with remainder about the real point \(x\). There exists \(\xi=\xi(\eta,\tau)\in\Omega\) such that
\begin{align}
	f(x+\eta+i\tau)
	= f(x) + (\eta+i\tau)f'(x) + \frac{(\eta+i\tau)^2}{2} f''(x)
	+ \frac{(\eta+i\tau)^3}{6} f^{(3)}(\xi).
\end{align}
	Taking imaginary parts and dividing by \(\tau\) (for \(\tau>0\)) gives
\begin{align}
		\frac{\Im[f(x+\eta+i\tau)]}{\tau}
	= f'(x) + \eta f''(x) + \frac{1}{2}\eta^2 f^{(3)}(\xi) - \frac{1}{6}\tau^2 f^{(3)}(\xi).
\end{align}
	Integrate this identity in \(\tau\in[\varepsilon,Z]\) and then in \(\eta\in[-H,H]\). The integrand admits a removable singularity at \(\tau=0\) (by analyticity), hence the improper limit \(\varepsilon\to0^+\) exists; interchanging limits and integrals is justified by the uniform bound on \(f^{(3)}\).
	
	Performing the \(\tau\)-integration and dividing by \(Z\) yields, for each fixed \(\eta\),
\begin{align}
	\frac{1}{Z}\int_{0}^{Z}\frac{\Im[f(x+\eta+i\tau)]}{\tau}\,d\tau
	= f'(x) + \eta f''(x) + \frac{1}{2}\eta^2 f^{(3)}(\tilde\xi(\eta)) - \frac{1}{18} Z^2 f^{(3)}(\hat\xi(\eta)),
\end{align}
	for suitable points \(\tilde\xi(\eta)$, with $\hat\xi(\eta)\in\Omega\) (mean-value forms of the remainder). Now average over \(\eta\in[-H,H]\) with normalization \(1/(2H)\). The term linear in \(\eta\) vanishes by symmetry:
\begin{align}
		\frac{1}{2H}\int_{-H}^{H} \eta \, d\eta = 0,
\end{align}
	while the constant \(f'(x)\) remains unchanged. The averaged quadratic remainder is bounded by
\begin{align}
		\left|\frac{1}{2H}\int_{-H}^{H}\frac{1}{2}\eta^2 f^{(3)}(\tilde\xi(\eta))\,d\eta\right|
	\le \frac{1}{2H}\int_{-H}^{H}\frac{1}{2}\eta^2 M\,d\eta
	= \frac{M}{6} H^2,
\end{align}
	since \(\int_{-H}^{H}\eta^2\,d\eta = 2 H^3/3\). The \(\tau\)-remainder contributes at most
\begin{align}
		\left|\frac{1}{2H}\int_{-H}^{H}\frac{1}{18} Z^2 f^{(3)}(\hat\xi(\eta))\,d\eta\right|
	\le \frac{M}{18} Z^2.
\end{align}
	Combining these bounds yields eq. \eqref{eq:csit-symmetric-remainder} and completes the proof.
\end{proof}

\begin{remark}
	The theorem shows that, for analytic \(f\), the  CSIT \(\mathcal{C}_{H,Z}\) recovers the first derivative to leading order with no \(\error(H)\) bias; the dominant truncation errors are \(\error(H^2)\) and \(\error(Z^2)\). In practice it is therefore natural to choose \(H\) and \(Z\) of comparable size \(\error(\Delta x)\) to obtain an overall second-order truncation error while maintaining the smoothing/regularization properties of the operator.
\end{remark}

\subsection{Operator Properties}

Assume $f$ is analytic in the symmetric rectangle
\begin{align}
	\Omega := \{\,x+\eta+i\tau : -H \le \eta \le H, \; 0 \le \tau \le Z\,\},
\end{align}
and satisfies the uniform bounds
\begin{align}
	\sup_{z\in \Omega} |f^{(3)}(z)| \le M, \qquad
	\sup_{x-H\le y \le x+H} |f'(y)| \le A, \qquad
	\sup_{x-H\le y \le x+H} |f''(y)| \le B.
\end{align}

Then the Complex-Step Integral Transform $\mathcal{C}_{H,Z}$ defined in eq. \eqref{eq.Complex_step_transform} is well-defined and satisfies the following properties:
\begin{itemize}
	\item \textbf{Linearity:}
	For any $\alpha,\beta \in \mathbb{C}$ and functions $f,g$ satisfying the hypotheses,
\begin{align}
	\mathcal{C}_{H,Z}(\alpha f + \beta g)
	= \alpha\,\mathcal{C}_{H,Z} f + \beta\,\mathcal{C}_{H,Z} g,
\end{align}
	directly from the linearity of the integral and of the limit $\varepsilon\to0^+$.
	
	\item \textbf{Translation invariance:} For any $x_0 \in \mathbb{R}$,
\begin{align}
	\mathcal{C}_{H,Z}[f(\cdot - x_0)](x) = \mathcal{C}_{H,Z} f(x - x_0),
\end{align}
5
	by change of variable in both $\eta$ and $\tau$.
	
	\item \textbf{Boundedness:} 	From Theorem~\ref{thm:csit-symmetric} and the uniform bounds on $f',f'',f^{(3)}$,
\begin{align}
		\|\mathcal{C}_{H,Z} f\|_\infty
	\;\le\;
	A \;+\; \frac{M}{6}\,H^2 \;+\; \frac{M}{18}\,Z^2.
\end{align}
	In particular, $\mathcal{C}_{H,Z}$ acts as a bounded linear operator from this analytic–$C^3$ space into $L^\infty(\mathbb{R})$.
	
	\item \textbf{Derivative-like nature:} Expanding $f$ locally by Taylor’s theorem and applying the symmetric integral shows
\begin{align}
	\mathcal{C}_{H,Z} f(x) = f'(x) + \widetilde{E}(x;H,Z), \qquad |\widetilde{E}(x;H,Z)| \le \frac{M}{6}H^2 + \frac{M}{18}Z^2.
\end{align}
	Thus $\mathcal{C}_{H,Z}$ approximates the first derivative with a second-order truncation error in both parameters \(H\) and \(Z\).
\end{itemize}

\section{Applications}

\subsection{Analytical Examples}

To develop basic understanding of the introduced integral transform, we next apply it to classical analytic functions. Table~\ref{tb.CSIT_common_Functions} summarizes the results, that have been obtained using explicit calculations presented in \citet{schaums_76} and \citet{king_2009}, 
together with the classical sine- and hyperbolic-sine integrals defined as
\begin{align}
	\mathrm{Shi}(z) = \int_0^z \frac{\sinh t}{t}\,dt, 
	\qquad 
	\mathrm{Si}(z) = \int_0^z \frac{\sin t}{t}\,dt.
	\label{eq.Shi_Si_function}
\end{align}
As we will next see, these functions provide compact analytic representations for the $\tau$-integration in the CSIT, 
while the $\eta$-integration over $[-H,H]$ acts as a symmetric spatial average that removes bias and introduces smooth multi-scale filtering.

\paragraph{Sine and cosine functions}

For $f(x)=\sin x$ or $f(x)=\cos x$, the CSIT behaves as a regularized, smoothed derivative operator—similar in phase to the Hilbert transform but scaled by the factor $\mathrm{Shi}(Z)$ and the symmetric averaging kernel (see Fig. \ref{fig:CSIT_Fourier}):
\[
(\mathcal{C}_{H,Z}\sin x)(x)
=
\mathrm{Shi}(Z)\,\frac{\sin(H)}{H}\,\cos x,
\qquad
(\mathcal{C}_{H,Z}\cos x)(x)
=
-\mathrm{Shi}(Z)\,\frac{\sin(H)}{H}\,\sin x.
\]
In the limit $H\to0$, $\sin(H)/H \to 1$, recovering the classical $\pi/2$ phase shift of the Hilbert transform.

\paragraph{Exponential function}

For $f(x)=e^x$, the $\tau$-integration yields the sine integral $\mathrm{Si}(Z)$, and the symmetric $\eta$-integration provides a finite difference average:
\[
(\mathcal{C}_{H,Z}e^x)(x)
=
\frac{\mathrm{Si}(Z)}{2H}\big(e^{x+H} - e^{x-H}\big),
\]
showing that $\mathcal{C}_{H,Z}$ remains well-defined even when the Hilbert transform diverges.

\paragraph{Gaussian pulse}

For $f(x)=e^{-x^2}$, analytic continuation through the error function gives
\[
(\mathcal{C}_{H,Z}e^{-x^2})(x)
=
\frac{1}{2H}\!\int_{-H}^H
\!\!\left[-\frac{\sqrt{\pi}}{4}
\big(
\mathrm{erf}(x+\eta-iZ)
+\mathrm{erf}(x+\eta+iZ)
-2\,\mathrm{erf}(x+\eta)
\big)\right]d\eta,
\]
illustrating both its derivative-like effect and its smoothing nature.

\paragraph{Complex Exponential}

For the complex exponential $f(x) = e^{i x}$, both the Hilbert transform and the CSIT act as phase-shifting operators. 
However, their amplitude behavior and regularity differ fundamentally.

The Hilbert transform yields
\begin{align}
	\mathcal{H}[e^{i x}] = -\,i\, e^{i x},
\end{align}
which represents a \emph{pure quadrature phase shift} of $-\pi/2$ without any amplitude modification. In contrast, the CSIT produces
\begin{align}
	\mathcal{C}_{H,Z}[e^{i x}] = i\,\mathrm{Shi}(Z)\,\frac{\sin(H)}{H}\, e^{i x},
\end{align}
which maintains the same $\pi/2$ phase shift (via the factor $i$), but includes an additional \emph{real-valued amplitude modulation} (see Fig. \ref{fig:CSIT_Fourier}) given by 
\begin{align}
	A_{H,Z} = \mathrm{Shi}(Z)\,\frac{\sin(H)}{H}.
\end{align}
This amplitude factor regularizes the high-wavenumber behavior and introduces smooth, scale-dependent attenuation.

Hence, while the Hilbert transform acts as an exact, non-dissipative phase operator, the CSIT generalizes it into a \emph{regularized, multi-scale analytic continuation}: it retains the correct phase relation ($\pi/2$ shift) but applies gentle smoothing through the $\sin(H)/(H)$ kernel in real space and the hyperbolic-sine integral $\mathrm{Shi}(Z)$ in complex space.

Note that in the asymptotic limit $H,Z \to 0$, the CSIT converges to the Hilbert transform,
\begin{align}
	\lim_{H,Z \to 0} \mathcal{C}_{H,Z}[e^{i x}] = -\,i\,e^{i x}
\end{align}
demonstrating that, for this case, the CSIT serves as a regularized, scale-aware extension of the classical Hilbert operator.

\begin{table}
	\footnotesize
	\caption{Comparison between the CSIT $\mathcal{C}_{H,Z}$ and the Hilbert transform for classical analytic functions.}
	\label{tb.CSIT_common_Functions}
	\begin{center}
		{\tabulinesep=1.5mm
			\begin{tabu}{ | p{2.6cm} | p{6.3cm} | p{5.5cm} | }
				\hline
				\textbf{Function} 
				& \textbf{CSIT $\mathcal{C}_{H,Z}$ (symmetric form)} 
				& \textbf{Hilbert Transform} 
				\\ \hline
				$\sin x$ 
				& $\displaystyle \mathrm{Shi}(Z)\,\frac{\sin(H)}{H}\,\cos x$ 
				& $-\cos x$ 
				\\ \hline
				$\cos x$ 
				& $\displaystyle -\,\mathrm{Shi}(Z)\,\frac{\sin(H)}{H}\,\sin x$ 
				& $\sin x$
				\\ \hline
				$e^{x}$ 
				& $\displaystyle \frac{\mathrm{Si}(Z)}{2H}\big(e^{x+H} - e^{x-H}\big)$ 
				& does not exist 
				\\ \hline
				$e^{-x^2}$ 
				& $\displaystyle \frac{1}{2H}\!\int_{-H}^H\!-\frac{\sqrt{\pi}}{4}\big[\mathrm{erf}(x+\eta-iZ)+\mathrm{erf}(x+\eta+iZ)-2\mathrm{erf}(x+\eta)\big]\,d\eta$ 
				& $\imag\pi\,\mathrm{erf}(\imag x)\,e^{-x^2}$ 
				\\ \hline
				$e^{i x}$ 
				& $\displaystyle i \frac{\sin(H)}{H}\,\mathrm{Shi}(Z)\,e^{i x}$ 
				& $-i\,e^{i x}$ 
				\\ \hline
		\end{tabu}}
	\end{center}
\end{table}
\begin{figure}
	\begin{center}
		\includegraphics[width=.8\textwidth]{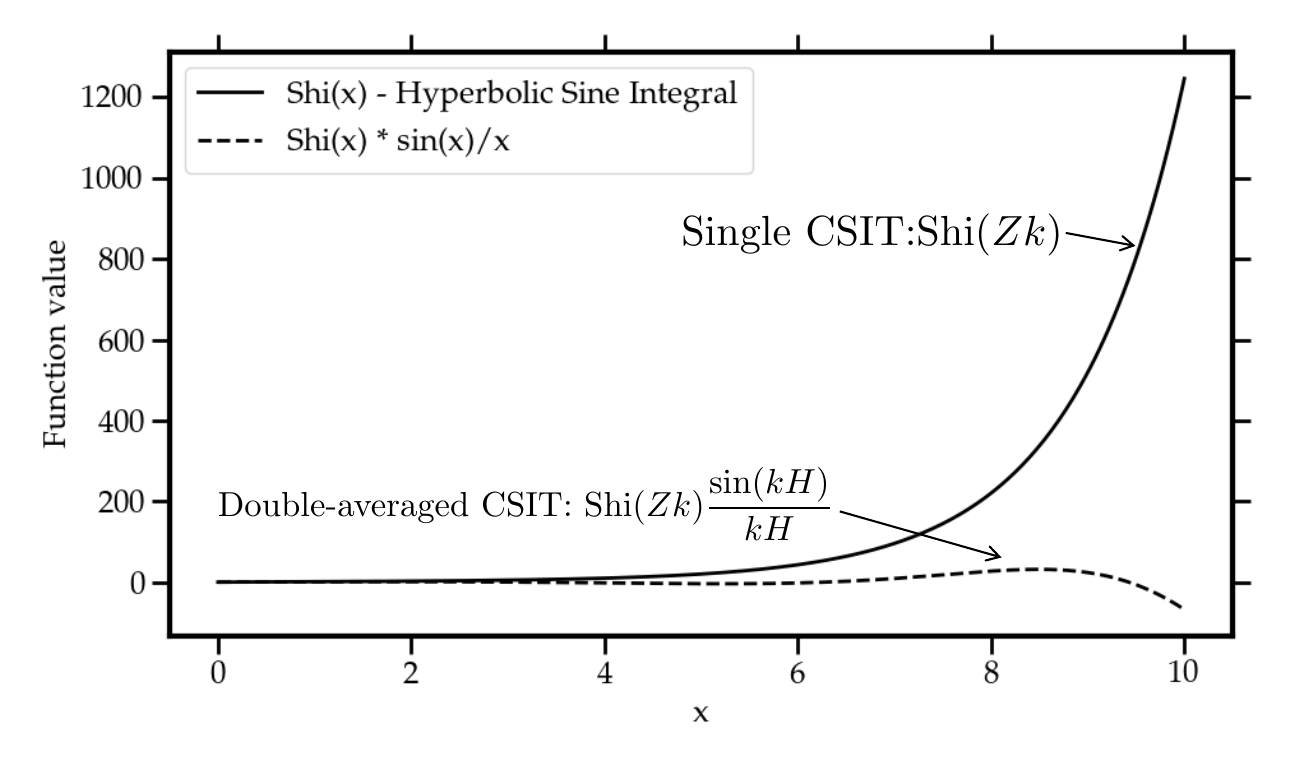}
		\caption{Fourier amplitude response of the CSIT operator. Single CSIT $(H=0)$ uses only the $\tau$ integral, approximated by $\mathrm{Shi}(Z k)$. Double-averaged CSIT includes $\eta \in [-H,H]$ averaging, introducing the $\sin(kH)/(kH)$ modulation, which damps high-wavenumber modes.}
		\label{fig:CSIT_Fourier}
	\end{center}
\end{figure}

\subsection{Practical Computation of $\Im[f(x + \eta + i\tau)]$}

For most practical applications we will not have access to the analytical function of the data $f(x)$. In these cases, the double-integral CSIT can be efficiently computed using the fast Fourier transform (FFT) as follows:
\begin{enumerate}
	\item Compute the discrete (discrete) FFT of the real-valued data \( f_j \):
	\[
	f_j \;\longrightarrow\; \hat{f}_k = \mathrm{FFT}(f_j) .
	\]
	\item Multiply each Fourier coefficient by the complex shift factor:
	\[
	\hat{f}_k \;\longrightarrow\; \hat{f}_k \, e^{i k \eta} e^{- k \tau},
	\]
	corresponding to a shift by \(\eta\) in the real direction and analytic continuation \(i\tau\) in the imaginary direction.
	\item Apply the inverse FFT to recover the shifted function on the real spatial grid:
	\begin{align}
			f(x + \eta + i\tau) = \mathrm{FFT}^{-1}[\hat{f}_k] . 
			\label{eq.Im_FFT}
	\end{align}
	\item Finally extract the imaginary part:
	\[
	\Im[f(x + \eta + i\tau)] .
	\]
\end{enumerate}
To compute the CSIT integral one can apply a suitable numerical quadrature (e.g., trapezoidal or composite Simpson rule) over both variables. Fine sampling near $\tau = 0$ ensures accuracy due to the $\tau^{-1}$ weight.

Note that the FFT is used to implement the \emph{analytic continuation} efficiently for all pairs $(\eta,\tau)$. Conceptually, CSIT does not require FFT; direct evaluation is possible but less efficient and accurate for large grids and many quadrature points.

\subsection{Finite Difference vs. Complex-Step with FFT vs. Pseudospectral Derivative}

\subsubsection{Finite Difference (FD) Operations}

When applying finite-difference methods to compute derivative approximations, the parameter \( h \) represents a \emph{spatial shift} along the computational grid. This means that the derivative approximations are computed using function values such as \( f(x+h) \) and \( f(x-h) \), which correspond to different grid points. Therefore, the step size \( h \) must be an integer multiple of the grid spacing \( \Delta x \). Numerically, this means accessing neighboring elements of the discrete array that represents the function $f(x)$. The resulting derivative approximation is therefore a purely local, grid-aligned operation.

\subsubsection{Complex-Step Method with FFT Operations}

In the complex-step formulation using the FFT, the parameters \( h \) and \( v \) do \emph{not} correspond to discrete grid shifts. Instead, the function \( f(x) \) is sampled at discrete grid points, and the FFT provides its spectral (Fourier) representation, i.e., 
\begin{align}
	f(x) = \sum_k c_k \, e^{i k x},
\end{align}
where \(c_k\) are the Fourier coefficients. A complex shift in physical space, \(x \to x + h + i v\), is then implemented by a spectral multiplication:
\begin{align}
	f(x + h + i v) = \sum_k c_k \, e^{i k h} e^{-k v}.
\end{align}
Multiplying by \( e^{i k h} e^{-k v} \) corresponds to an \emph{analytic continuation} of this Fourier series. Consequently, the FFT allows one to evaluate \( f(x) \) at any complex location \( x + h + i v \), not only at grid-aligned points as it is done using FD. In simpler words, the parameters \( h \) and \( v \) can take arbitrarily small continuous values and need not correspond to grid indices. The parameters $h$ and $v$  are free continuous parameters in the FFT complex-step method and they have different roles:
\begin{itemize}
	\item \textbf{Imaginary shift \( v \):}  it controls the complex-step derivative accuracy. It can be chosen arbitrarily small (e.g.\ \( v \approx 10^{-10} \))—the smaller the value, the higher the precision (until numerical round-off errors dominate), since the difference quotient is not affected by subtractive cancellation as in finite differences.
	\item \textbf{Real shift \( h \):} If \( h = 0 \), the evaluation occurs exactly at the original grid points and if \( h \neq 0 \), the evaluation is performed between grid points. The accuracy depends on how well the FFT representation captures the underlying function (i.e., the spectral resolution).
\end{itemize}

\subsubsection{CSIT with FFT vs.\ Pseudospectral Derivative}

The pseudospectral derivative is defined as \citep{igel2017computational}:
\begin{align}
	f'(x) = \mathrm{FFT}^{-1} \!\left( i k \, \mathrm{FFT}[f](k) \right),
	\label{eq.pseudospectral_derivative}
\end{align}
where the derivative is exact for periodic and sufficiently smooth functions. In contrast, the CSIT formulation modifies the spectral factor \(i k\) by the regularized kernel implied by the imaginary shift \(v\):
\begin{align}
	f'_{\text{CSIT}}(x) = \mathrm{FFT}^{-1} \!\left(e^{i k \eta} e^{- k \tau} \, \mathrm{FFT}[f](k) \right).
\end{align}
For small \(h,v\), the CSIT converges to the pseudospectral derivative. However, for finite \(h,v\), high-wavenumber components are damped, leading to a natural smoothing effect.

The key distinction lies in stability and spectral selectivity. The pseudospectral derivative reproduces the exact derivative but amplifies high-frequency numerical noise, especially near discontinuities or at domain boundaries. The CSIT derivative, on the other hand, introduces a controlled exponential damping proportional to \(e^{-k v}\) that regularizes the high-frequency modes while preserving spectral accuracy in the low- and mid-wavenumber range. In this sense, CSIT can be interpreted as a \emph{regularized pseudospectral derivative}: it retains the spectral precision of the Fourier operator while suppressing parasitic oscillations and edge artifacts through analytic continuation.

The smoothing is not imposed externally but arises naturally from the complex-step formulation, making the CSIT derivative particularly robust in non-periodic or heterogeneous settings where the standard pseudospectral approach becomes unstable, as we will next see.

\subsubsection{Illustrative Example: Logistic Function}

To illustrate the differences between the finite-difference, pseudospectral, and CSIT derivatives, we consider the logistic function
\begin{align}
	f(t) = \frac{1}{1 + \exp\!\big[-k (t - t_0)\big]},
	\label{eq.Logistic_function}
\end{align}
where $t$ denotes time, $t_0$ is the midpoint of the curve, and $k$ controls the steepness of the transition. In this example, we choose $t \in [0,1]$~s, $t_0 = 0.5$~s, $k = 100$, and a total number of samples $N = 500$.

The derivative of Eq.~\eqref{eq.Logistic_function} is computed using three methods: the centered finite-difference (FD) scheme, the pseudospectral derivative (Eq.~\eqref{eq.pseudospectral_derivative}), and the CSIT formulation (Eq.~\eqref{eq.Complex_step_transform}). The results are shown in Fig.~\ref{fig:CSIT_Derivatives}.

\begin{figure}
	\centering
	\includegraphics[width=\textwidth]{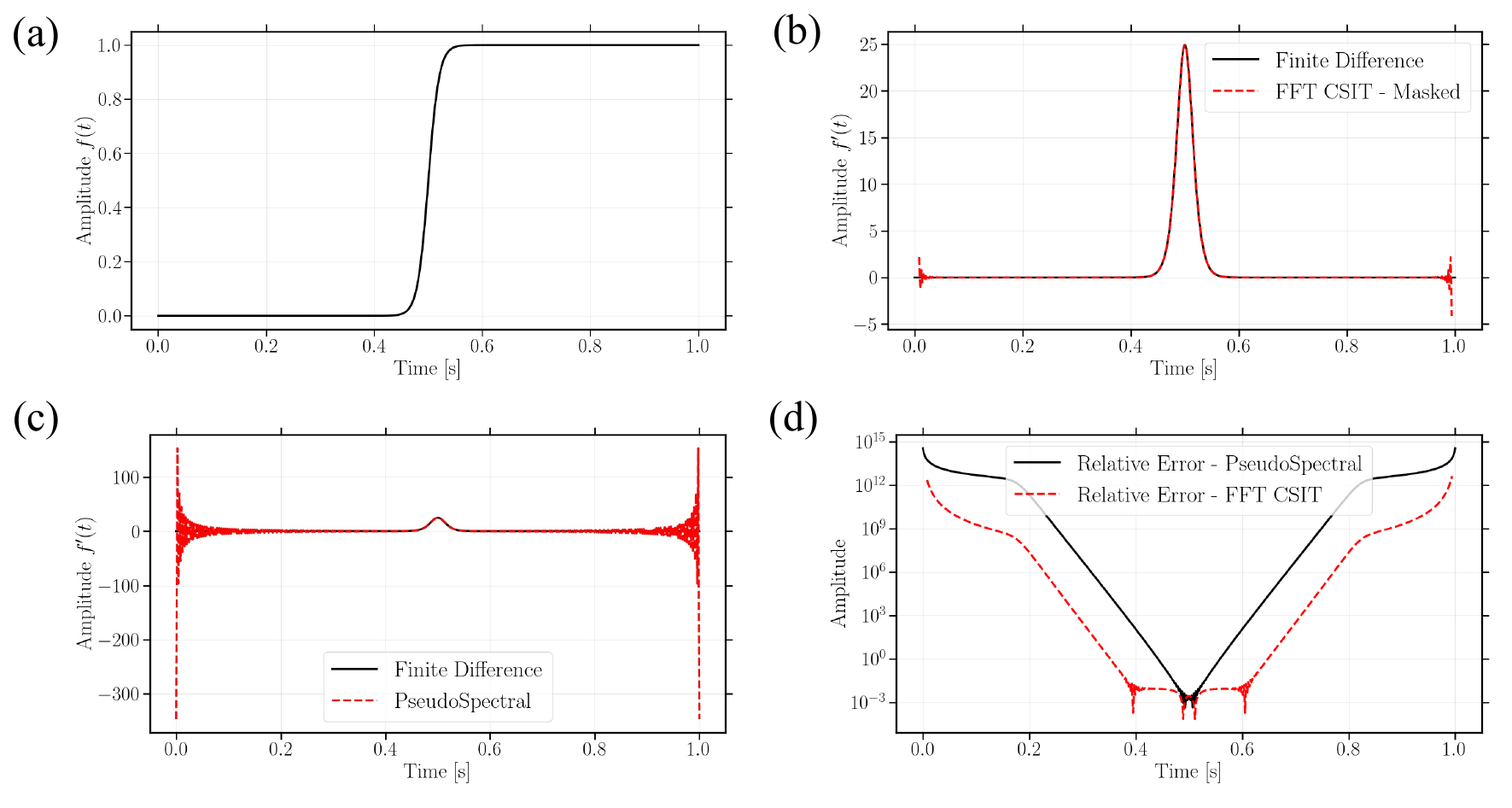}
	\caption{(a) Logistic function (Eq.~\eqref{eq.Logistic_function}) for $t\in[0,1]$~s, $t_0=0.5$~s, $k=100$, and $N=500$. (b) Comparison between the centered finite-difference and CSIT derivatives. (c) Pseudospectral derivative. (d) Relative errors with respect to the analytical derivative.}
	\label{fig:CSIT_Derivatives}
\end{figure}

As expected, the finite-difference derivative exhibits a smooth behavior, free of spurious high-frequency components, but with limited resolution near the steep transition region. In contrast, the pseudospectral derivative (Fig.~\ref{fig:CSIT_Derivatives}c) emphasizes high-frequency modes, resulting in visible oscillations and numerical noise. This behavior arises because the logistic function changes rapidly over a very narrow region. Accurately representing this sharp transition in a Fourier basis requires a large number of high-wavenumber components, which dominate the spectral derivative and amplify any small errors or truncation effects.

The CSIT approach, on the other hand, modifies the Fourier coefficients by a complex shift factor $e^{i k \eta} e^{- k \tau}$, effectively performing a symmetric average (see Eq.~\eqref{eq.Complex_step_transform}). This operation acts as a natural spectral regularization, selectively damping the contribution of high-wavenumber components without the need for explicit filtering. In this example, we set $h = v = \Delta t$, which is sufficient to suppress high-frequency oscillations while preserving the correct derivative amplitude (Fig.~\ref{fig:CSIT_Derivatives}b).

Relative errors are shown in Fig.~\ref{fig:CSIT_Derivatives}d, where we observe that the CSIT derivative is consistently more accurate than the Fourier derivative, except near the midpoint of the function. This behavior can be understood as follows: at $t = t_0$, the logistic function is smooth and locally symmetric, so its derivative is perfectly represented by the available Fourier modes. In this region, the pseudospectral derivative attains its natural spectral accuracy, while the CSIT introduces a small averaging effect due to the complex shifts $h$ and $v$, slightly smoothing the peak derivative. Thus, the CSIT provides improved stability and accuracy in regions of rapid variation, while the pure Fourier derivative remains optimal where the function is locally linear and symmetric.

It is important to note that, as in any FFT-based formulation, the assumption of periodicity introduces potential boundary artifacts. When evaluating the imaginary part of the complex-shifted function, discontinuities between the end points of the interval can cause aliasing and distortions near the boundaries. These artifacts arise because the averaging implied by the CSIT becomes asymmetric close to the edges, where the periodic extension no longer represents the true function behavior. Consequently, we recommend applying the CSIT derivative only to the interior points—typically excluding four to five grid points near each boundary—analogous to how centered finite-difference schemes omit boundary nodes to maintain symmetry and accuracy.

\subsection{CSIT Solution of the Advection Equation}

We next illustrate the application of the introduced CSIT to the one-dimensional advection equation given by the following expression
\begin{align}
	\partial_t u(x,t) + c(x) \partial_x u(x,t) = f(t) \delta(x - x_0), 
	\qquad x \in [0, L], \quad t > 0,
	\label{eq.Advection_equation}
\end{align}
where $u(x,t)$ represents the transported field (e.g., displacement, temperature, or concentration), $c$ denotes the advection speed, $f(t)$ is a prescribed time-dependent source, and $\delta(x - x_0)$ is the Dirac delta distribution introducing a localized excitation at position $x_0$.

In this formulation, the CSIT operator involves the evaluation of the imaginary component $\Im[f(x + \eta + i\tau)]$, which is computed efficiently using the FFT as previously explained. This spectral implementation makes the Pseudospectral method a natural point of comparison for assessing the accuracy and stability of the CSIT-based discretization. As will be shown, the CSIT formulation inherits the spectral accuracy of the FFT while providing a natural regularization mechanism that mitigates the parasitic (checkerboard) modes commonly observed in classical pseudospectral advection schemes.

\subsubsection{Numerical Discretizations}

We compare the CSIT-based solution of the advection equation with two classical numerical schemes: the finite-difference (FD) and the pseudospectral (PS) methods. 
This comparison highlights the consistency, stability, and parasitic-mode behavior of the CSIT derivative relative to well-known discretizations.

\paragraph{Finite-difference scheme}
The leapfrog finite-difference discretization of the advection equation~\eqref{eq.Advection_equation} reads
\begin{align}
	\frac{u_j^{t+\Delta t} - u_j^{t-\Delta t}}{2\Delta t}
	= 
	c \, \frac{u_{j+1}^{t} - u_{j-1}^{t}}{2\Delta x}
	+ f(t)\,\delta(x_j - x_0)
	+ \mathcal{O}(\Delta x^2, \Delta t^2),
	\label{eq.FDM}
\end{align}
where $u_j^t = u(x_j, t)$ and $j$ indexes the spatial grid. This second-order central scheme is well known to generate dispersive ``parasitic'' modes for high wavenumbers.

\paragraph{Pseudospectral scheme}
The pseudospectral discretization computes the spatial derivative in Fourier space, yielding
\begin{align}
	\frac{u_j^{t+\Delta t} - u_j^{t-\Delta t}}{2\Delta t}
	= 
	c \, \mathrm{PS}_x[u^t_j]
	+ f(t)\,\delta(x_j - x_0)
	+ \mathcal{O}(\Delta t^2),
	\label{eq.PSM}
\end{align}
where the pseudospectral derivative operator is defined in eq. \eqref{eq.pseudospectral_derivative}. This method achieves spectral accuracy for periodic functions but may amplify numerical oscillations if the grid or forcing is not smooth.

\paragraph{CSIT-based discretization}
Finally, we use the CSIT as a generalized spatial derivative operator:
\begin{align}
	\frac{u_j^{t+\Delta t} - u_j^{t-\Delta t}}{2\Delta t}
	= 
	c \, \mathrm{CSIT}[u_j^t]
	+ f(t)\,\delta(x_j - x_0)
	+ \mathcal{O}(\Delta x, v^2, \Delta t^2).
	\label{eq.CSIT_Advection}
\end{align}
The CSIT effectively replaces the derivative by a smoothed analytic continuation, mitigating spurious oscillations and non-physical energy transfer between modes.

\subsubsection{Results}

Simulation parameters are listed in Table~\ref{tb.material_parameters}. All schemes use the same time step and spatial grid; only the derivative operator differs. For the CSIT implementation, we set $h = 0.0005\,\Delta x$ and $v = 0.1\,\Delta x$ after convergence testing.
\begin{table}
	\caption{Simulation parameters for the 1D advection equation.}
	\label{tb.material_parameters}
	\begin{center}
		{\renewcommand{\arraystretch}{1.5}
			\begin{tabular}{ | c | c | c | c | c | c | }
				\hline 
				Velocity [m/s] & $L_x$ [m] & $x_s$ [m] & $f_0$ [Hz] & $n_x$ & $c \Delta t / \Delta x$ \\ 
				\hline 
				900 & 10000 & 5000 & 1 & 500 & 0.25 \\ 
				\hline    
		\end{tabular}}
	\end{center}
\end{table}

Figure~\ref{fig:CSIT_Simulations} compares the three solutions. Both the finite-difference and pseudospectral schemes exhibit parasitic modes in the form of secondary oscillations trailing the main pulse, particularly at early times. In contrast, the CSIT-based solution maintains a clean, physically consistent wavefront, demonstrating its ability to suppress non-physical components while preserving the correct phase velocity. 
\begin{figure}
	\begin{center}
		\includegraphics[width=1\textwidth]{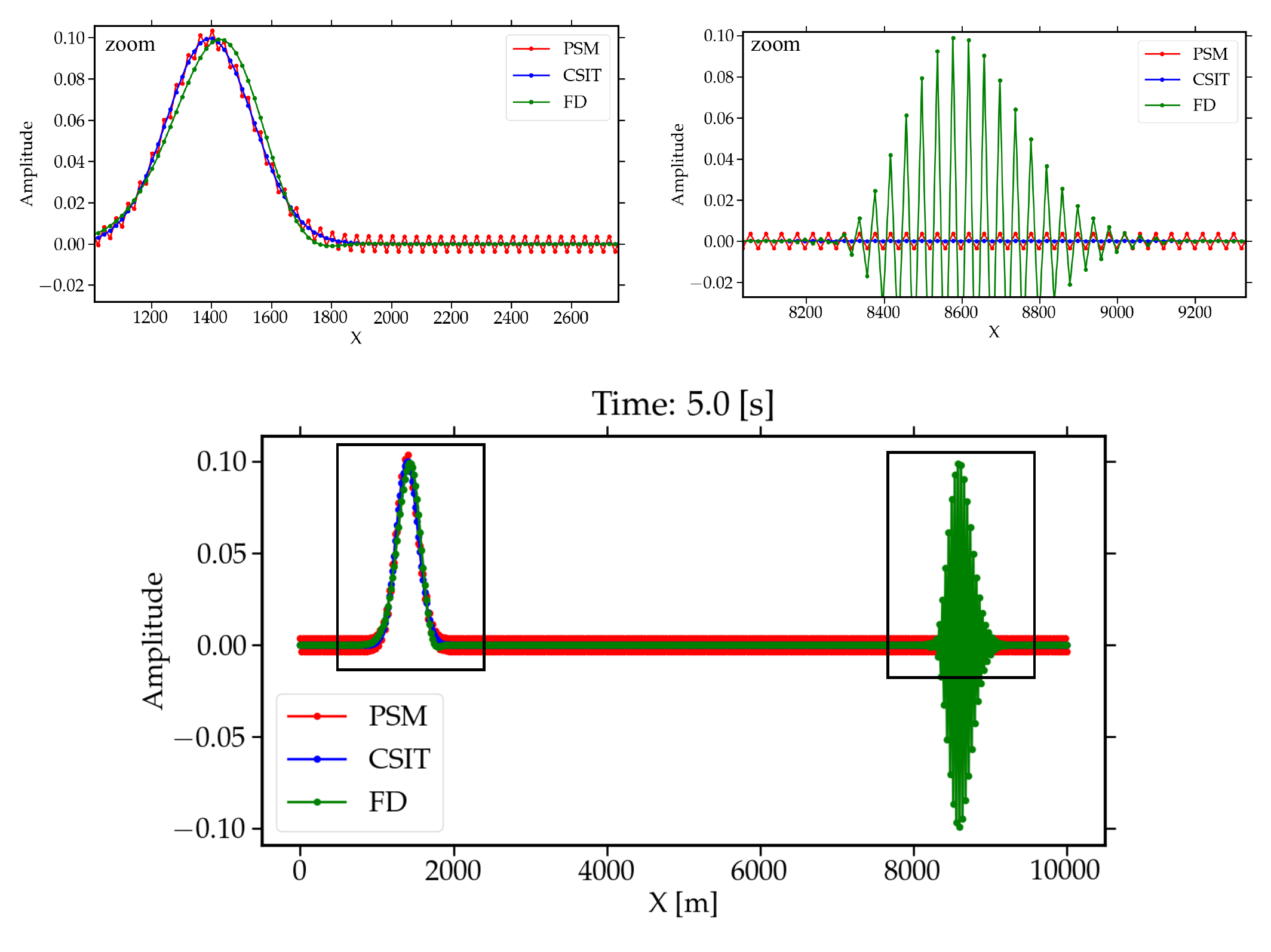}
		\caption{Comparison of advection solutions using finite-difference, pseudospectral, and CSIT discretizations. The CSIT solution eliminates the parasitic mode observed in the other two schemes.}
		\label{fig:CSIT_Simulations}
	\end{center}
\end{figure}

\subsubsection{Dispersion Analysis}

To better understand the observed suppression of parasitic modes, we analyze the numerical dispersion relation associated with each spatial discretization. 
For a harmonic solution of the form $u(x,t) = e^{i(kx - \omega t)}$, substituting into the semi-discrete forms of the advection equation yields the following relations between the numerical frequency $\omega$ and wavenumber $k$.

\paragraph{Fourier and finite-difference schemes}

For the pseudospectral discretization, the numerical derivative in Fourier space is exact, giving $\omega = c k$. While this relation is perfectly non-dispersive in theory, in practice, finite resolution and periodic aliasing produce high-wavenumber artifacts that appear as parasitic oscillations in the numerical solution. In contrast, the centered finite-difference scheme introduces the discrete dispersion relation
\begin{align}
	\omega_{\mathrm{FD}} = \frac{c}{\Delta x}\sin(k \Delta x),
\end{align}
which leads to phase errors and the familiar trailing or leading dispersive waves.

\paragraph{CSIT derivative}

The CSIT operator modifies the derivative symbol through the complex-analytic continuation and symmetric averaging, leading to
\begin{align}
	\omega_{\mathrm{CSIT}} = c \, k \, \mathrm{Shi}(Z k) \frac{\sin(k H)}{k H}.
\end{align}
The multiplicative factors $\mathrm{Shi}(Zk)$ and $\sin(kH)/(kH)$ act as smooth low-pass filters: the first damps high-$k$ components exponentially through the imaginary shift $Z$, while the second attenuates them algebraically through the symmetric averaging over $\eta \in [-H,H]$. 
As a result, the phase velocity $v_\phi = \Re[\omega_{\mathrm{CSIT}}]/k$ remains close to $c$ for low and mid wavenumbers, while the amplitude of spurious short-wavelength components is efficiently reduced.

\paragraph{Single CSIT ($H=0$)}

An important and somewhat unexpected result is that the CSIT removes parasitic modes even in its simplest form, \emph{without any real averaging} ($H=0$). When only the imaginary continuation is used ($Z>0$, $H=0$), the CSIT modifies the wavenumber symbol of the derivative by the factor $\mathrm{Shi}(Zk)$, leading to
\begin{align}
	\omega_{\mathrm{CSIT}} = c\,k\,\mathrm{Shi}(Zk).
\end{align}
The hyperbolic-sine integral $\mathrm{Shi}(Zk)$ behaves as $\mathrm{Shi}(Zk)\approx 1$ for $|kZ|\ll1$ and grows sublinearly for large $|kZ|$, effectively damping the amplitude of high-wavenumber modes. This soft filtering originates from the analytic continuation in the imaginary direction and thus provides a \emph{built-in spectral regularization}, even without real-space averaging.

\paragraph{Physical interpretation}

Unlike artificial viscosity or explicit smoothing filters, the damping introduced by $\mathrm{Shi}(Zk)$ arises naturally from the analytic structure of the complex integral transform. The imaginary displacement $Z$ acts as a continuous filter in the complex plane, attenuating components associated with rapidly oscillating, non-analytic extensions of $f(x)$. This mechanism explains why the CSIT eliminates parasitic oscillations even in the single-shift configuration.

\subsubsection{Pseudocode: 1D Advection with Symmetric CSIT-Based Derivative}

We next show the pseudocode used in the previous simulations for the CSIT solution of the advection equation. 
\begin{algorithm}
	\caption{1D Advection with Symmetric CSIT Derivative}
	\begin{algorithmic}[1]
		\State \textbf{Input:} Grid $x_j$, initial condition $u(x,0)$, speed $c$, time step $\Delta t$, max imaginary shift $Z$, max real shift $H$, number of $\tau$ points $M$, number of $\eta$ points $P$, number of time steps $N_t$
		\State \textbf{Precompute:} FFT wavenumbers $k_j = 2\pi j / L$ (for uniform grid)
		\State \textbf{Discretize:} $\tau_m = m \Delta \tau$, $m=1,\dots,M$, $\Delta \tau = Z / M$
		\State \textbf{Discretize:} symmetric $\eta_p = -H + (p-0.5) \Delta \eta$, $p=1,\dots,P$, $\Delta \eta = 2H / P$
		\For{each time step $n=1$ to $N_t$}
		\State Initialize $u_x \gets 0$ 
		\For{each $\eta_p$}
		\For{each $\tau_m$}
		\State Compute FFT: $\hat u = \text{FFT}(u)$
		\State Analytic continuation with real shift $\eta_p$ and imaginary shift $\tau_m$: 
		\[
		\hat u_{\eta_p, \tau_m} = \hat u \cdot \exp(i k \eta_p - k \tau_m)
		\]
		\State Inverse FFT: $u_{\rm complex} = \text{IFFT}(\hat u_{\eta_p, \tau_m})$
		\State Take imaginary part: $u_{\rm im} = \operatorname{Im}(u_{\rm complex})$
		\State Increment derivative integral: $u_x \gets u_x + w_m w_p \cdot (u_{\rm im}/\tau_m)$ \Comment{Trapezoidal weights $w_m$, $w_p$}
		\EndFor
		\EndFor
		\State Normalize by $2ZH$: $u_x \gets \frac{1}{2 Z H} u_x$
		\State Update $u$ using chosen time-stepping scheme, e.g., explicit Euler:
		\[
		u^{n+1} = u^n - c \Delta t \, u_x
		\]
		\EndFor
		\State \textbf{Output:} $u(x, t_{\rm final})$
	\end{algorithmic}
\end{algorithm}

\subsection{Regularized Calculation of the Instantaneous Frequency}

The complex trace $v(t)$ is defined as \citep{taner1979complex}
\begin{align}
	v(t) = x(t) + i y(t) = x(t) + i \mathscr{H} x(t) =  A(t) e^{i\theta(t)} ,
	\label{eq.analytic_signal}
\end{align}
where $x(t)$ is the time series analyzed, $\mathscr{H}$ is the Hilbert transform, $A$ the amplitude and $\theta$ the (instantaneous) phase, that can be obtained from \citep{bracewell1986fourier}
\begin{align}
	\theta(t) = \tan^{-1} \left(\frac{y(t)}{x(t)}\right) .
\end{align} 
The instantaneous frequency $f^{\text{inst}}(t)$ is defined as the rate of change of the instantaneous phase as follows \citep{barnes1991instantaneous}
\begin{align}
	f^{\text{inst}}(t) = \frac{1}{2\pi} \partial_t \theta(t) ,
\end{align}
which results in
\begin{align}
	f^{\text{inst}}(t) = \frac{1}{2\pi} \frac{x(t)\partial_t y(t)-y(t)\partial_tx(t)}{x^2(t)+y^2(t)} .
	\label{eq.Instantaneous_Frequency}
\end{align}
However, when \( x(t),y(t) \to 0 \), the denominator in eq.~\eqref{eq.Instantaneous_Frequency} tends to zero, leading to indeterminate values and large numerical spikes that are physically meaningless \citep[e.g.,][]{matheney1995seismic}.  This effect can be avoided by artificially adding a small damping factor as follows
\begin{align}
	f^{\text{inst}}(t) = \frac{1}{2\pi} \frac{x(t)\partial_t y(t)-y(t)\partial_tx(t)}{x^2(t)+y^2(t)+\epsilon^2}
	\label{eq.Damped_Instantaneous_Frequency}
\end{align}
where \( \epsilon \) acts as a regularization parameter \citep{matheney1995seismic}. The damped instantaneous frequency in eq.~\eqref{eq.Damped_Instantaneous_Frequency} has been widely used to estimate seismic attenuation in the Earth \citep[e.g.,][]{matheney1995seismic,Ford2012,durand2013insights}. Several alternative regularization strategies have been proposed to mitigate the singularity in eq.~\eqref{eq.Instantaneous_Frequency} when \( x,y \to 0 \) \citep[e.g.,][]{fomel2007local}. In the following, we show that the CSIT provides a natural and stable alternative, avoiding any artificially imposed regularization.

\subsubsection{The Instantaneous Frequency using the CSIT: a Natural Regularization}

We next show how the application of the CSIT avoids any indetermination and any artificially imposed regularization in the calculation of the instantaneous frequency. We can compute the instantaneous frequency $f^{\text{inst}}$ using the CSIT as follows
\begin{align}
\begin{aligned}
		f^{\text{inst}}(t) & = \frac{1}{2\pi} \partial_t \theta(t) = \frac{1}{2\pi} \partial_t \left[\arctan \left(\frac{y(t)}{x(t)}\right)\right] \\
		& \approx \frac{1}{2\pi}\lim_{\varepsilon\to0^+}\frac{1}{2HZ}\int_{-H}^H\int_{\varepsilon}^Z  \frac{\Im \left[\arctan\left(\frac{y+\eta+i\tau}{x+\eta+i\tau}\right)\right]}{\tau} 
	\,d\tau\,d\eta 
\end{aligned}
	\label{eq.CSIT_Instantaneous_Frequency}
\end{align}
The integrand in eq. \eqref{eq.CSIT_Instantaneous_Frequency} are can be written as follows
\begin{align}
	\begin{aligned}
		\frac{\Im \left[\arctan\left(\frac{B}{A}\right)\right]}{v}  &=  \frac{1}{2v} \ln\left( \frac{A-iB}{A+iB} \right),
	\end{aligned}
	\label{eq.CS_Instantaneous_Frequency}
\end{align}
where 
\begin{align}
	A(t)= x(t)+h+iv, \quad B(t)=y(t)+h+iv
\end{align}
with $h,v \in \mathbb{R}^+$ and $h,v\to 0$. The instantaneous frequency thus becomes
\begin{align}
	\begin{aligned}
		f^{\text{inst}}(t) \approx \frac{1}{2\pi}\lim_{\varepsilon\to0^+}\frac{1}{4HZ}\int_{-H}^H\int_{\varepsilon}^Z  \frac{\ln\left[\frac{(x+\eta+i\tau)-i(y+\eta+i\tau)}{(x+\eta+i\tau)+i(y+\eta+i\tau)} \right] }{\tau} 
		\,d\tau\,d\eta .
	\end{aligned}
	\label{eq.CSIT_Instantaneous_Frequency_Analitical}
\end{align}

\subsubsection{Understanding the Complex-Step Parameters}

Note that we make an imaginary perturbation of both $x(t)$ and $y(t)$, i.e., $x(t)\to x(t)+h+iv$, in the Fourier domain. This means that we multiply the Fourier representation of the signal $ X(\omega)$ by the spectral multiplier $e^{i \omega h - \omega v}$ as follows
\begin{align}
	X(\omega) \to X(\omega) e^{i \omega h - \omega v} .
\end{align}

\paragraph{The Real Shift $h$}

The term $e^{i \omega h}$ corresponds to a \textit{real time shift} of the signal by $h$ seconds
\begin{align}
	x(t+h) \longleftrightarrow X(\omega) e^{i \omega h}.
\end{align}

This real shift $h$ shifts the phase slightly and helps avoid purely numerical derivatives at exact sample points. In general terms, $h$ slightly shifts the derivative point in time for numerical stability.

\paragraph{The Imaginary Shift $v$}
The term $e^{-\omega v}$ represents an \textit{imaginary-time shift}, which is the key to the \emph{complex-step derivative}:
\begin{align}
	x(t + i v) \longleftrightarrow X(\omega) e^{- \omega v}.
\end{align}

It acts as a regularization when the signal amplitude approaches zero and prevents division by zero in eq. \eqref{eq.CS_Instantaneous_Frequency}. It also smooths spikes in the instantaneous frequency at amplitude zeros or phase flips. In general terms, $v$ injects a tiny damping along the imaginary axis to avoid singularities in $B/A$ when $A \sim 0$.

\subsubsection{No Need to Compute the Hilbert Transform}

The Hilbert Transform in the Fourier Domain can be written as follows
\begin{align}
\mathcal{F}\{H[x(t)]\}(\omega) = -i \, \operatorname{sgn}(\omega) \mathcal{F}\{x(t)\}= -i \, \operatorname{sgn}(\omega) \, X(\omega),
\end{align}
where $\operatorname{sgn}$ is the sign function. Therefore, the analytic signal (eq. \eqref{eq.analytic_signal}) in the Fourier Domain can be written as follows
\begin{align}
Z(\omega) = X(\omega) + i \, \mathcal{F}\{H[x(t)]\}(\omega) 
= X(\omega) + i (-i \, \operatorname{sgn}(\omega) X(\omega)) 
= X(\omega) + \operatorname{sgn}(\omega) X(\omega) .
\end{align}

As a consequence there is no need to compute the Hilbert transform to compute the instantaneous frequency using the CSIT.

\subsubsection{An Analytical Test}

We now examine how the finite-difference (FD; eq.~\eqref{eq.Instantaneous_Frequency}) and CSIT (eq.~\eqref{eq.CSIT_Instantaneous_Frequency_Analitical}) approximations to the instantaneous frequency behave as a function of the sampling density of the signal. As a test case, we consider an analytical chirp defined as
\begin{align}
	f(t) = \RealPart \left[\exp\left(2\pi i \left(f_0 t + 0.5\, k_c\, t^2 \right)\right)\right], 
	\qquad t \in [0,1],
	\label{eq.chirp_signal}
\end{align}
where the initial frequency is $f_0 = 20$~Hz and the chirp rate is $k_c = 20$. The corresponding analytical instantaneous frequency is
\begin{align}
	f_{\text{true}}^{\text{inst}}(t) = f_0 + k_c t .
	\label{eq.True_IF}
\end{align}

To analyze the effect of temporal sampling, we discretize the interval $[0,1]$ using two different sampling densities, $n_t = 2500$ and $n_t = 300$. Fig.~\ref{fig:IF_analytical} shows the reconstructed analytic signal for both cases together with the instantaneous frequency estimated using the FD and CSIT formulations.
\begin{figure}
	\begin{center}
		\includegraphics[width=1\textwidth]{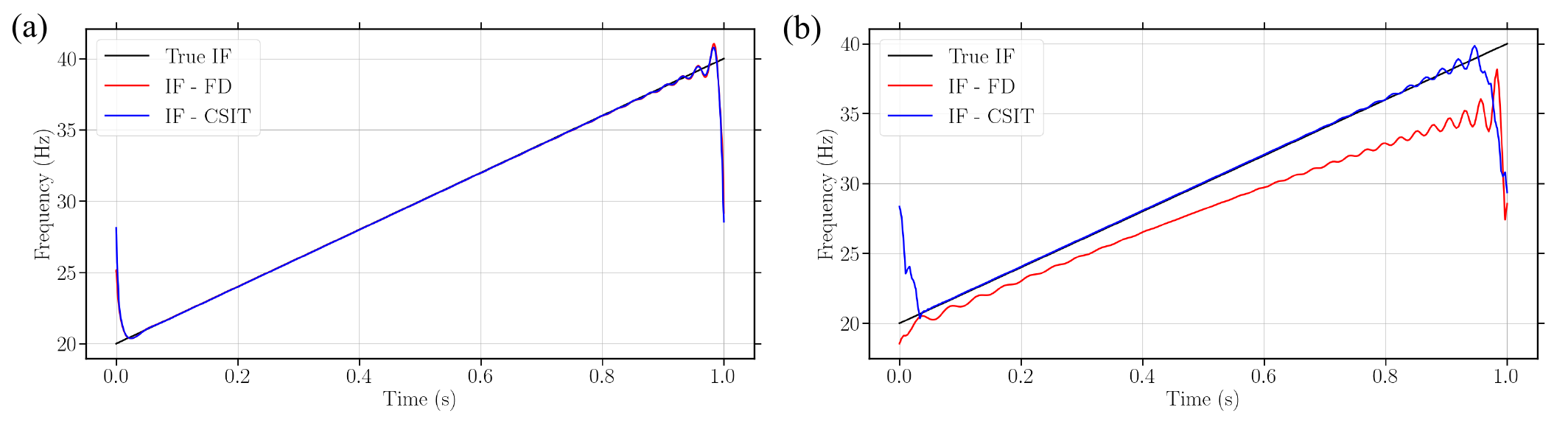}
		\caption{Reconstruction of the instantaneous frequency for the chirp signal defined in eq.~\eqref{eq.True_IF} using two time-sampling rates $n_t$: (a) $n_t = 2500$ and (b) $n_t = 300$.}
		\label{fig:IF_analytical}
	\end{center}
\end{figure}

For a sufficiently large number of samples ($n_t \rightarrow \infty$), both methods converge to the analytical instantaneous frequency. The small errors near the boundaries are identical for both approaches and arise from the Fourier transform used to construct the analytic signal in the case of the CSIT and for computing the Hilbert transform in the case of FD. Although such edge effects can be mitigated through smoothing or tapering, we intentionally avoid any pre-processing in order to isolate the effects of sampling density rather than Fourier-domain aliasing. When the sampling density is substantially reduced ($n_t = 300$), the CSIT approximation (eq.~\eqref{eq.CSIT_Instantaneous_Frequency_Analitical}) still achieves higher accuracy compared to the classical FD method (eq.~\eqref{eq.Instantaneous_Frequency}). 

This highlights two important characteristics of the CSIT formulation when applied to instantaneous frequency estimation:
\begin{enumerate}
\item The CSIT maintains high accuracy even under coarse sampling, in contrast to the finite-difference approach whose sensitivity to grid spacing leads to pronounced errors as the sampling rate decreases. This robustness arises because CSIT performs differentiation in the Fourier domain, where derivative operators are represented exactly and are less susceptible to truncation and subtraction errors. 
\item Both methods share similar boundary artifacts, which are inherent to the construction of the analytic signal via the Fourier transform rather than to the derivative operator itself. The superior performance of CSIT in the interior of the domain suggests that it can be particularly advantageous for seismological applications involving sparse observations or signals with rapidly varying frequency content, where traditional finite-difference estimates become unreliable.
\end{enumerate}

The analytical results above demonstrate that the CSIT formulation remains accurate and stable even under coarse sampling, and that it systematically suppresses spurious oscillations that arise in finite-difference estimates of the instantaneous frequency. While these controlled experiments isolate the numerical behavior of the method, real seismological signals introduce additional challenges, including noise contamination, amplitude modulations, complex wave interference, and the practical limitations of instrument sampling rates. To evaluate the performance of the CSIT under such realistic conditions, we next apply the method to broadband seismic data from the 2010 deep-focus Granada earthquake. This example provides a test case: the signal contains sharp amplitude variations and multiple overlapping phases, conditions under which finite-difference–based instantaneous frequency estimates often become unstable. The following section examines how the CSIT behaves in this setting and assesses its advantages over conventional approaches when applied to real Earth data.

\subsubsection{Applications to Seismological Data}

On April 11, 2010, a significant seismic event with a moment magnitude of 6.3 Mw occurred in the Granada region of Spain. The earthquake, originating at a depth of approximately 650 kilometers, is considered to be a rare deep-focus event. Its considerable depth suggests that it was an isolated occurrence, distinct from the shallow crustal seismicity that typically characterizes the region \citep{buforn20112010}. Deep-focus earthquakes of this kind provide unique opportunities to investigate mantle dynamics and subduction-related processes at great depths, offering valuable constraints for models of the Earth's interior \citep[e.g.][]{buforn20112010, bezada2012contrasting, sun2024revealing}.

We downloaded seismological data using ObsPy \citep{Krischer2015} from the seismological station ASYE (latitude $11.56^{\circ}$, longitude $41.44^{\circ}$). Fig.~\ref{fig:Granada_event}a shows the location of the Granada earthquake (red star) and the selected station (blue triangle). For the instantaneous frequency analysis, we focus on the vertical component of the normalized ground velocity (Fig.~\ref{fig:Granada_event}b). Standard data pre-processing was performed using Obspy \citep{Krischer2015} and included band-pass filtering between 1--10s.
\begin{figure}
	\begin{center}
		\includegraphics[width=1\textwidth]{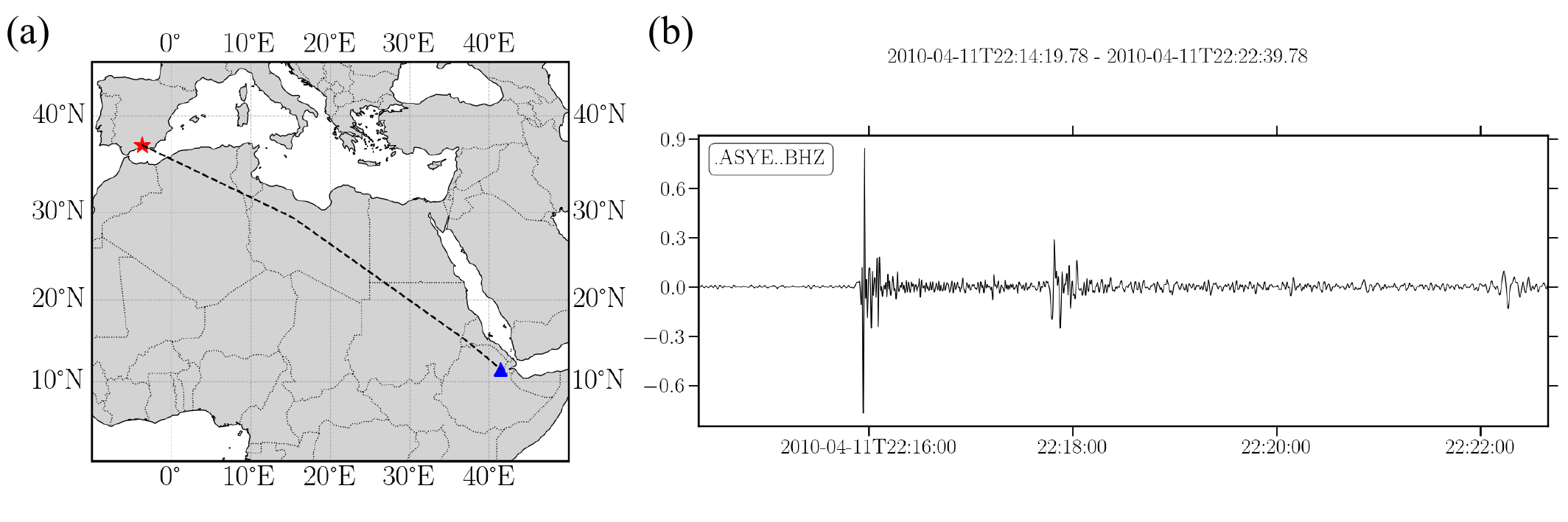}
		\caption{(a) Location of the April 2010 Granada earthquake epicenter (red star) and seismological station used in this study (blue triangle). (b) Normalized vertical component velocity recording of the event.}
		\label{fig:Granada_event}
	\end{center}
\end{figure}

For comparison, we compute the instantaneous frequency using conventional finite-difference techniques (eq. \eqref{eq.Instantaneous_Frequency}) and using the CSIT (eq. \eqref{eq.CSIT_Instantaneous_Frequency_Analitical}). For the calculation of the CSIT, we selected the parameters $H=Z=\Delta t$, $\epsilon=10^{-2}\Delta t$ with $N_{\eta}=4$, $N_{\tau}=4$. It is important not to choose $\epsilon$ too small, as excessively small values amplify high-wavenumber noise inherent to the analytic continuation and lead to spurious oscillations.

Results are presented in Fig.~\ref{fig:CSIT_DATA}. At small amplitudes (Figs.~\ref{fig:CSIT_DATA}a,b), both methods produce nearly identical instantaneous frequencies. However, the very large spikes produced by the finite-difference approximation---particularly during rapid amplitude variations---are absent in the CSIT estimate. The separated panels (Figs.~\ref{fig:CSIT_DATA}c,d) illustrate this distinction clearly: the finite-difference method exhibits high-amplitude behavior, whereas the CSIT result remains stable and physically interpretable. This behavior arises because the CSIT naturally regularizes the derivative through its complex shift, suppressing artifacts associated with large amplitude gradients and reducing sensitivity to local noise.
\begin{figure}
	\centering
	\includegraphics[width=\textwidth]{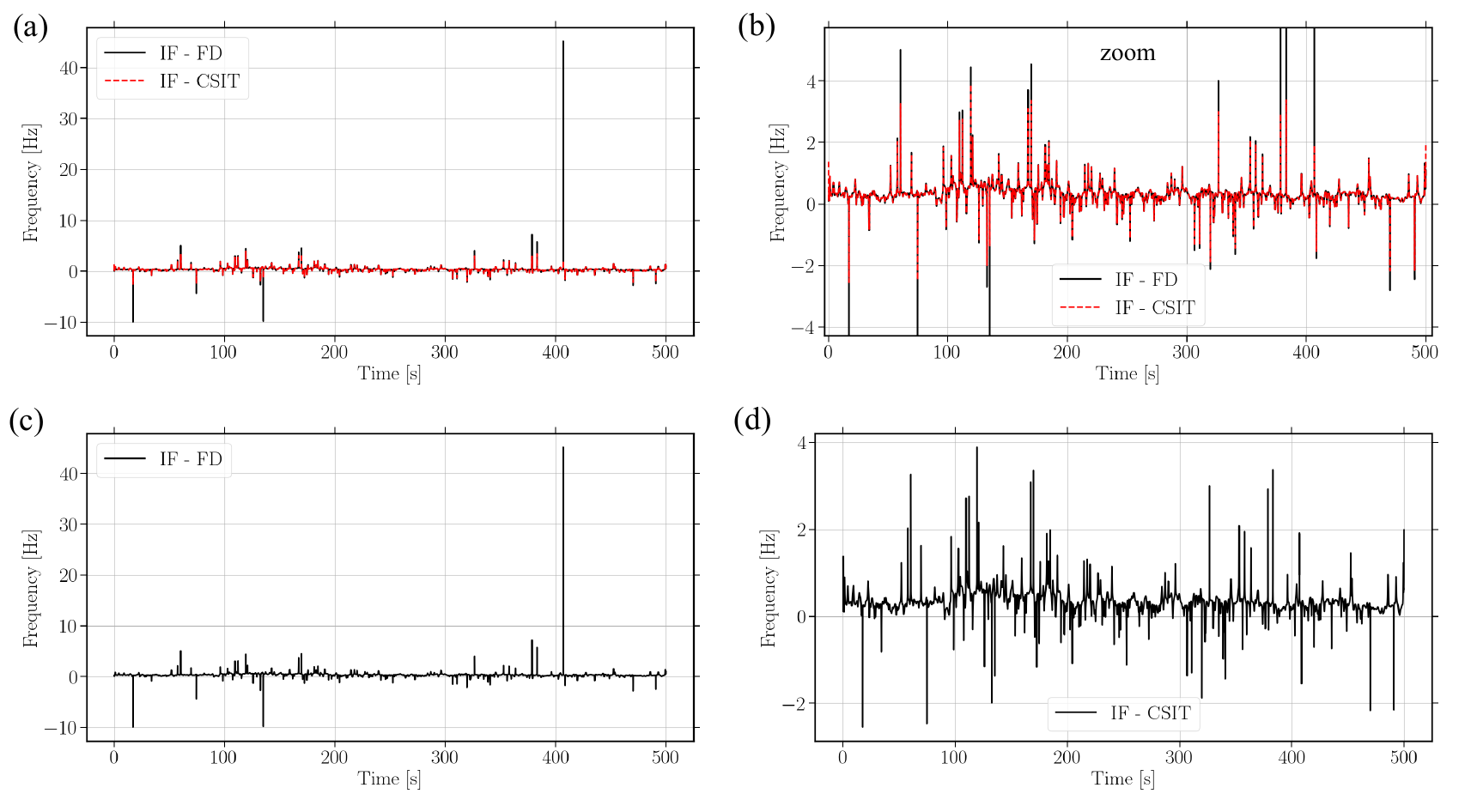}
	\caption{(a) Instantaneous frequency of the Granada 2010 event computed using finite-difference techniques (eq.~\ref{eq.Instantaneous_Frequency}) and the CSIT (eq.~\ref{eq.CSIT_Instantaneous_Frequency_Analitical}). (b) Zoom of panel (a). (c) Instantaneous frequency obtained using the finite-difference method. (d) Instantaneous frequency obtained using the CSIT.}
	\label{fig:CSIT_DATA}
\end{figure}

\section{Discussion}

\subsection{On the Origins of the Complex-Step Method: The Stieltjes Inversion Formula}

The Complex-Step method, originally introduced by \citet{Squire1998} and later generalized by \citet{Abreu201384}, has its theoretical foundations/connections in the Stieltjes inversion formula (circa 1880--1890). The Stieltjes inversion formula provides a means of recovering a measure (or distribution function) from its Stieltjes transform.

\begin{definition}[Stieltjes Transform]
	Let $\mu$ be a positive, finite measure on the real line. The \emph{Stieltjes transform} of $\mu$ is defined as \citep{anderson2010introduction,HirschmanWidder1955,WidderLaplaceTransform}
	\begin{align}
		S_{\mu}(z) =  \int_{-\infty}^{\infty} \frac{d\mu(t)}{t - z}, \quad z \in \mathbb{C} \setminus \mathbb{R}.
	\end{align}
\end{definition}

For $z \in \mathbb{C} \setminus \mathbb{R}$, both the real and imaginary parts of $(t-z)^{-1}$ are continuous functions of $t \in \mathbb{R}$, and are absolutely integrable with respect to any finite measure~$\mu$. The Stieltjes inversion formula then allows one to recover the original measure~$\mu$ from its transform~$S_\mu(z)$.

\begin{definition}[Stieltjes Inverse Transform]
	For any open interval $I\in[a,b]$ whose endpoints are not atoms of $\mu$ (that is, points carrying a positive, nonzero amount of measure) \citep{anderson2010introduction,HirschmanWidder1955,WidderLaplaceTransform}, the measure of the interval is recovered as
	\begin{align}
		\begin{aligned}
			\mu((a,b]) 
			&= \lim_{\epsilon \to 0^+} \frac{1}{\pi} \int_I \frac{S_\mu(x + i\epsilon) - S_\mu(x - i\epsilon)}{2i} \, dx \\
			&= \lim_{\epsilon \to 0^+} \frac{1}{\pi}  \int_I \Im \, S_\mu(x + i\epsilon) \, dx .
			\label{eq.Stieltjes_Inverse_Formula}
		\end{aligned}
	\end{align}
\end{definition}

The Stieltjes inversion formula thus provides a direct way to reconstruct the measure~$\mu$ from its transform~$S_{\mu}(z)$. Stieltjes' work was part of the broader development of functional analysis and measure theory that we will not discuss here. We only note that the Stieltjes transform and its inverse remain central tools in several fields of science, including spectral theory, quantum mechanics, random matrix theory, and the theory of orthogonal polynomials.

From a modern numerical perspective, the conceptual basis of the Complex-Step method introduced by \citet{Squire1998} can be viewed as closely related to the analytic continuation underlying the Stieltjes inversion formula in eq.~\eqref{eq.Stieltjes_Inverse_Formula}. In the same spirit, the Complex-Step Integration Technique (CSIT) introduced in this work can be regarded as a generalization of eq.~\eqref{eq.Stieltjes_Inverse_Formula}, grounded in the complex-step formulation developed by \citet{Abreu201384}.

\subsection{Frequency-Domain Interpretation of the CSIT}

To understand the effect of the CSIT operator in spectral space, consider the Fourier representation of a periodic or analytically continuable function,
\[
f(x) = \sum_{k \in \mathbb{R}} \hat{f}(k) e^{i k x}.
\]
Substituting into the definition of the CSIT gives
\[
(\mathcal{C}_{H,Z} f)(x)
= \sum_k \hat{f}(k)\, \sigma_{H,Z}(k)\, e^{i k x},
\]
where the \textbf{Fourier multiplier (symbol)} of the CSIT operator is
\begin{equation}
	\sigma_{H,Z}(k)
	= i k \, \mathrm{Shi}(Zk)\, \frac{\sin(kH)}{kH}.
	\label{eq.csit_symbol}
\end{equation}

This expression reveals several key insights:

\begin{enumerate}
	\item \textbf{Derivative-like behavior.}  
	For small $kH$ and $kZ$, the expansions
	\[
	\mathrm{Shi}(Zk) = Zk + \tfrac{1}{18}(Zk)^3 + \mathcal{O}((Zk)^5), 
	\qquad 
	\frac{\sin(kH)}{kH} = 1 - \tfrac{1}{6}(kH)^2 + \mathcal{O}((kH)^4),
	\]
	show that
	\[
	\sigma_{H,Z}(k) = i k \left[ 1 + \mathcal{O}(Z^2k^2 + H^2k^2) \right],
	\]
	confirming that the CSIT approximates the derivative operator $d/dx$ to second order in both $H$ and $Z$.
	
	\item \textbf{Amplitude modulation.}  
	The factor $\frac{\sin(kH)}{kH}$ introduces a smooth low-pass filtering effect, attenuating high-wavenumber components ($|k|H \gg 1$).  
	This explains the CSIT’s ability to suppress parasitic oscillations or grid-scale noise in discrete PDE simulations.
	
	\item \textbf{Vertical complex-step damping.}  
	The $\mathrm{Shi}(Zk)$ term grows monotonically for real $Zk>0$, reflecting the exponential continuation of $f$ into the complex plane.  
	For small $Z$, it acts as a controlled amplification of the derivative magnitude without phase distortion.
	
	\item \textbf{Phase preservation.}  
	Since $\sigma_{H,Z}(k)$ is purely imaginary for real $k$, the CSIT preserves phase relationships and introduces no dispersive phase error—only amplitude modulation.
\end{enumerate}

Thus, the CSIT may be viewed as a \textit{spectrally filtered derivative operator}:
\[
\mathcal{C}_{H,Z} 
= 
\mathcal{F}^{-1} \!\left[\, i k\, \mathrm{Shi}(Zk)\, \frac{\sin(kH)}{kH} \,\mathcal{F}[\cdot] \right],
\]
where $Z$ controls the vertical complex-step smoothing and $H$ controls the horizontal averaging.

\subsection{Choosing $H$ and $Z$ relative to the grid spacing $\Delta x$}

When implementing the CSIT on discrete data with grid spacing $\Delta x$, the truncation parameters $H$ (real-direction averaging) and $Z$ (imaginary-step height) should be selected to balance accuracy, stability, and smoothing. Both parameters parameters $H$ and $Z$, act as scale parameters controlling the operator's derivative approximation and spectral filtering strength.

\begin{enumerate}
	\item \textbf{Consistency order:}  
	The symmetric form of the CSIT satisfies
\begin{align}
	(\mathcal{C}_{H,Z} f)(x) = f'(x) + \mathcal{O}(H^2) + \mathcal{O}(Z^2),
\end{align}
 provided $f^{(3)}$ is bounded in the complex neighborhood $\Omega_{H,Z}$. Thus, choosing $H \sim Z \sim \mathcal{O}(\Delta x)$ yields an overall truncation error  of second-order $\error(\Delta x^2)$.
	
	\item \textbf{Smoothing and scale separation:}  
	The $\eta$-averaging introduces a low-pass filter $\mathrm{sinc}(kH)$, while the vertical complex-step produces the $\mathrm{Shi}(Zk)$ modulation. To preserve physical gradients while suppressing grid-scale oscillations, $H$ and $Z$ should be of smaller order as the local grid spacing. Excessively large $H$ or $Z$ overly damps high frequencies.
	
	\item \textbf{Numerical stability and precision:}  
	Unlike conventional finite differences, the complex-step formulation avoids cancellation errors for small $\tau$ \citep{Squire1998}. The integrand remains regular as $\tau \to 0$, so smaller $Z$ can be used safely down to machine precision. 
	
	\item \textbf{Bias control via symmetry:}  
	The asymmetric form, i.e., when $\eta\in[0,H]$, introduces a linear bias $\tfrac{H}{2} f''(x)$.	Using centered averaging over $\eta \in [-H,H]$ eliminates this term, leaving only second-order truncation errors. In numerical practice, this symmetric form is strongly recommended.
	
	\item \textbf{Quadrature and discretization:}  
	The double integral in $\eta$ and $\tau$ can be approximated by any desired quadrature. Since the integrand is smooth near $\tau=0$, no special treatment of singularities is required. Finer resolution in $\tau$ near zero can improve accuracy at negligible cost.
\end{enumerate}

\begin{remark}
	In time-dependent PDE solvers, the normalization factor $1/(2HZ)$ can be absorbed into a time-step coefficient or scaling constant. The theoretical scaling ensures that $\mathcal{C}_{H,Z} f \to f'$ as $H,Z \to 0$, while practical implementations may tune this value to balance derivative magnitude and damping.
\end{remark}

\subsection{Calculation of the CSIT using FFT vs. Interpolation Techniques}

In the present work, the CSIT has been computed using the FFT, which provides a natural and efficient way to evaluate analytic continuations of periodic or quasi-periodic fields. However, this is not a fundamental requirement: the same transform can be evaluated using real-space interpolation techniques. While the FFT-based approach achieves spectral accuracy for smooth domains, interpolation-based implementations may extend the method to heterogeneous or non-periodic settings. Their accuracy, however, depends on the polynomial order of the interpolation scheme and on the chosen bounds for $H$ and $Z$.

Table~\ref{tb.Fourier_Interpolated_CSIT} summarizes the main differences between both approaches in terms of accuracy, computational cost, and grid requirements.  The FFT-based CSIT acts as a \emph{global} interpolation operator, implicitly coupling all spatial modes. In contrast, polynomial or spline-based implementations would act \emph{locally}, potentially requiring the definition of a \textit{mesh of elements}, similar to finite-element or finite-volume frameworks. This local formulation increases implementation complexity, as it requires a consistent spatial topology and quadrature design, but it also opens the possibility of performing highly realistic and spatially heterogeneous simulations of partial differential equations (PDEs). 
A detailed exploration of this local, interpolation-based CSIT formulation will be the subject of future work.
\begin{table}[h!]
	\footnotesize
	\caption{Comparison of FFT-based and interpolation-based implementations of the CSIT.}
	\label{tb.Fourier_Interpolated_CSIT}
	\begin{tabularx}{\linewidth}{lXX}
		\toprule
		\textbf{Aspect} & \textbf{FFT-based CSIT} & \textbf{Interpolation-based CSIT} \\
		\midrule
		Accuracy & Spectrally accurate for smooth periodic functions; integrates naturally over $\eta$ and $\tau$ & Depends on interpolation order; accuracy may degrade for large $H$ or $Z$ \\
		Speed & Efficient via FFT/IFFT; double loop over $\eta$ and $\tau$ increases cost moderately & Slower; scales with grid size and number of interpolation and quadrature points \\
		Grid requirement & Uniform and periodic for standard FFT; adaptable with non-uniform FFT variants & Naturally supports non-uniform and non-periodic grids but requires careful interpolation \\
		Implementation & Compact and global; straightforward in spectral space & More complex; requires local topology definition and interpolation kernels \\
		\bottomrule
	\end{tabularx}
\end{table}

In addition, note that the spectral implementation of the CSIT inherits the exponential convergence of the Fourier representation for analytic functions, making it highly efficient for smooth, periodic problems. By contrast, interpolation-based CSIT will exhibit algebraic convergence, typically of order $\mathcal{O}(\Delta x^{p})$, where $p$ is the degree of the interpolation polynomial. Hence, for sufficiently smooth data, the FFT version is superior in accuracy per degree of freedom, whereas for strongly heterogeneous or non-smooth domains, local interpolation may provide more geometric flexibility at the cost of reduced convergence rate.

\subsection{CSIT vs Fourier Derivative in Advection Problems}

We have shown that the advection equation eq. \eqref{eq.Advection_equation} can be efficiently solved using the introduced CSIT, leading to a stable evolution that efficiently suppresses parasitic oscillations. In contrast, the classical pseudospectral scheme based on the Fourier derivative, while spectrally accurate for smooth periodic functions, is known to suffer from numerical ringing and spurious checkerboard modes when applied to under-resolved or non-periodic fields.

Table~\ref{tb:derivative_comparison_clean} summarizes the main differences between the Fourier derivative and various CSIT formulations. The single CSIT ($H=0$) already introduces mild regularization through the imaginary shift $i\tau$, acting as a built-in complex-step smoothing operator. Averaging over a finite real shift interval $\eta \in [0,H]$ further damps high-frequency components, while the symmetric version $\eta \in [-H,H]$ completely removes linear bias and provides robust, phase-consistent filtering without compromising the derivative-like character of the operator.
\begin{table}
	\footnotesize
	\caption{Comparison of Fourier derivative, single CSIT, double-averaged CSIT, and symmetric double-averaged CSIT.}
	\label{tb:derivative_comparison_clean}
	\begin{tabularx}{\linewidth}{lXXXX}
		\toprule
		Property & Fourier derivative & Single CSIT ($H=0$) & Double-averaged CSIT ($H>0$, $\eta \in [0,H]$) & Symmetric double-averaged CSIT ($H>0$, $\eta \in [-H,H]$) \\
		\midrule
		Taylor expansion & Exact for periodic functions, $f'(x)$ & $f'(x) + \mathcal{O}(Z^2)$ & $f'(x) + \frac{H}{2} f''(x) + \mathcal{O}(H^2 + Z^2)$ & $f'(x) + \mathcal{O}(H^2 + Z^2)$ \\
		Fourier factor & $i k$ & $\mathrm{Shi}(Z k)$ & $\mathrm{Shi}(Z k) \frac{\sin(kH)}{kH}$ & $\mathrm{Shi}(Z k) \frac{\sin(kH)}{kH}$ \\
		Effect on phase & None & $\pi/2$ Hilbert-like shift & $\pi/2$ shift, amplitude smoothed by $\sin(kH)/(kH)$ & Same as double-averaged, symmetric $\eta$ averaging removes linear bias \\
		Numerical bias & Neutral & Mild smoothing, derivative-like & Linear bias $\propto H/2$ & Linear bias removed by symmetric $\eta$ averaging \\
		Grid requirement & Uniform, periodic & Uniform preferred for FFT, non-uniform possible with interpolation & Same as single CSIT & Same as single CSIT \\
		Periodicity requirement & Required & Not required (analytic continuation) & Not required & Not required \\
		Implementation & FFT-based, simple & Requires FFT/IFFT with complex-step & Double loop over $\eta$ and $\tau$, 2D quadrature weights & Same as double-averaged \\
		Parasitic mode control & None & Mild suppression & Moderate suppression via [0,H] averaging & Strong suppression via symmetric $\eta$ averaging \\
		\bottomrule
	\end{tabularx}
\end{table}

In physical terms, the double-averaged CSIT can be interpreted as a multi-scale analytic continuation in the complex plane, combining the diffusive damping of the imaginary shift with the dispersive averaging of the real shift. This results in a smoother, bias-free approximation of the spatial derivative, which retains phase accuracy while naturally suppressing non-physical high-wavenumber modes.

From a computational perspective, CSIT maintains the spectral efficiency of FFT-based schemes but extends their applicability to non-periodic or weakly smooth domains. The symmetric averaging provides a tunable filter width $H$ that can be chosen proportionally to the grid spacing, balancing bias removal and mode suppression without introducing significant numerical diffusion.

In general, the symmetric double-averaged CSIT formulation represents an alternative to the Fourier derivative: it preserves spectral convergence while introducing intrinsic stabilization and bias control through analytic averaging in both real and imaginary directions.

\subsection{Interpretation of the CSIT-Based Instantaneous Frequency}

The CSIT instantaneous frequency does not merely compute a derivative of the phase like the classical Hilbert-based method; rather, it effectively performs a small complex-step \textit{probe} of the signal. This means that, instead of differentiating a noisy phase directly, CSIT measures how the complex phase behaves under a small complex perturbation. This arises due to the introduction of the imaginary shift $v$ and the real shift $h$, which provide several advantages:
\begin{itemize}
	\item \textbf{Complex smoothing:} The imaginary step $v$ acts as a regularized derivative over a small neighborhood in the complex plane, suppressing spurious spikes and small parasitic oscillations that often appear in the Hilbert-based instantaneous frequency.
	
	\item \textbf{Amplitude-aware:} The CSIT derivative naturally weights the calculation by the local amplitude. As a result, regions near maxima or envelope peaks are more clearly represented, yielding smoother and more accurate maxima compared to the sometimes noisy Hilbert-based results.
	
	\item \textbf{Bias removal:} The real shift $h$ acts as a local averaging parameter, reducing numerical artifacts near zero crossings without requiring any artificially imposed damping.
	
	\item \textbf{Local curvature interpretation:} Peaks in the original signal correspond to local phase curvature in the complex plane. The CSIT captures this more faithfully because it measures how the complex argument ``bends'' under small shifts, resulting in clearer representation of maxima and minima in the instantaneous frequency plot.
\end{itemize}

Therefore, rather than being a mere ``filtered version'' of the signal, the CSIT instantaneous frequency provides a physically meaningful regularization. It highlights the true phase dynamics, including maxima, minima, and local phase reversals, while effectively suppressing spurious numerical effects.

\section{Conclusions}

\paragraph{Theoretical Foundation and Generalization of the Complex-Step Method (CSIT)}  
	The Complex-Step Integral Transofrm (CSIT) combines the Complex-Step method proposed by \cite{Abreu201384} and the concept of the Hilbert transform. The CSIT provides a robust tool for numerical differentiation that avoids the pitfalls of traditional finite-difference methods, and particularly in the context of spectral methods. Its generalization through complex steps enables higher-order accuracy with fewer errors, particularly in the treatment of small perturbations in the complex plane.
	
\paragraph{Spectral Interpretation of the CSIT Operator}  
	The CSIT behaves like a derivative operator with the added benefits of spectral filtering and phase preservation. The operator exhibits derivative-like behavior, with a low-pass filtering effect that reduces parasitic oscillations and grid-scale noise. This makes CSIT particularly useful in simulations that involve high-frequency components, where traditional numerical differentiation methods can fail. The CSIT shows its potential to smooth out high-wavenumber components and provide accurate derivative approximations with minimal distortion when solving the advection equation.
	
\paragraph{Practical Considerations for Grid Spacing and Truncation Parameters}  
	The parameters $H$ and $Z$, with $H,Z\in\mathbb{R}^+$, control the real and imaginary steps, respectively, of the CSIT. They must be carefully chosen relative to the grid spacing to balance accuracy, stability, and computational efficiency. Proper tuning of these parameters ensures second-order accuracy while preventing numerical instabilities or excessive smoothing. A symmetric form of averaging over $\eta \in [-H,H]$ is recommended to eliminate bias and improve the overall robustness of the method.
	
\paragraph{Comparison of FFT and Interpolation Approaches}  
	The CSIT can be implemented using either FFT-based or interpolation-based methods, each with distinct advantages and trade-offs. The FFT-based approach provides spectral accuracy and computational efficiency, especially for periodic problems. However, for non-periodic or heterogeneous domains, interpolation-based implementations may offer greater flexibility, though at the cost of reduced convergence rate and increased computational complexity. Future work could explore the potential of interpolation-based CSIT formulations for non-periodic or more complex geometries and different numerical techniques such as the Finite-Element method is an ongoing work.
	
\paragraph{Advantages over Classical Fourier Derivatives in Advection Problems}  
	The CSIT provides a significant improvement over traditional Fourier derivatives, particularly in advection problems. The CSIT suppresses spurious oscillations and improves stability in under-resolved fields, where conventional Fourier methods suffer from numerical ringing. The double-averaged CSIT formulation, in particular, introduces controlled damping of high-frequency modes while preserving phase accuracy. This makes it a powerful tool for simulating physical systems governed by advection-dominated partial differential equations.
	
\paragraph{Improvement in Instantaneous Frequency Calculation}  
	The CSIT-based method for calculating instantaneous frequency provides a more reliable and physically meaningful representation of phase dynamics compared to classical Hilbert transform-based methods. By introducing both real and imaginary shifts, CSIT naturally smooths the signal and highlights important features such as local maxima and phase reversals. This makes it particularly advantageous for signal analysis in noisy or complex datasets, where the Hilbert transform may fail to provide accurate results.
	
\paragraph{Future Directions}  
	While the CSIT has proven effective in many contexts, further exploration is needed to assess its performance in more complex settings, such as non-uniform grids, heterogeneous materials, or irregular boundary conditions. Additionally, the development of interpolation-based CSIT formulations for non-periodic and spatially heterogeneous domains could open up new possibilities for applications in computational physics, engineering, and signal processing.

\section{Acknowledgments}

R.A. acknowledges initial constructive conversations with Angie Pineda and Fabian Bonilla.   

\section{Data availability}

Data used in this work are available and have been downloaded using Obspy \citep{Krischer2015} from IRIS. 

\footnotesize
\bibliographystyle{apalike}
\bibliography{Biblio}

\end{document}